\documentclass[11pt]{article}
\usepackage[left=1in,right=1in,top=1in, bottom=1in]{geometry}
\usepackage{amssymb, amsmath,amsthm}
\usepackage{mathrsfs}
\usepackage{amscd}
\usepackage[active]{srcltx}
\usepackage{verbatim}
\usepackage{graphicx}
\usepackage{hyperref}

\usepackage[utf8]{inputenc}
\usepackage[T1]{fontenc}
\usepackage{lmodern}
\usepackage[backend=bibtex]{biblatex}
\usepackage{hyperref} 
\usepackage{todonotes}
\usepackage{enumerate}
\usepackage{mathtools}
\usepackage{bbm}

\mathtoolsset{showonlyrefs,showmanualtags}

\newcommand\dd{\mathrm{d}}

\newcommand{\Oo}{\mathcal{O}}
\newcommand{\Pp}{\mathcal{P}}

\newcommand{\Ss}{\mathcal{S}}
\newcommand{\Tt}{\mathcal{T}}

\newcommand{\EE}{\mathbb{E}}

\newcommand{\NN}{\mathbb{N}}

\newcommand{\PP}{\mathbb{P}}

\newcommand{\RR}{\mathbb{R}}
\newcommand{\bS}{\mathbb{S}}

\newcommand{\inp}[2]{\langle #1,#2 \rangle}

\newcommand{\Exp}{\mathrm{Exp}}

\newcommand{\Cov}{\mathrm{Cov}}

\newcommand{\One}{\mathbf{1}}

\DeclareUnicodeCharacter{2113}{$\ell$}

\newcommand{\ip}[2]{\langle #1,#2\rangle}

\renewcommand{\epsilon}{\varepsilon}

\theoremstyle{plain}
\newtheorem{theorem}{Theorem}[section]
\theoremstyle{remark}

\newtheorem{example}[theorem]{Example}
\theoremstyle{plain}
\newtheorem{corollary}[theorem]{Corollary}
\newtheorem{lemma}[theorem]{Lemma}
\newtheorem{proposition}[theorem]{Proposition}
\newtheorem{definition}[theorem]{Definition}

\numberwithin{equation}{section}

\allowdisplaybreaks

\title{Quenched large deviations for randomly weighted geodesic random walks}

\author{Rik Versendaal\footnote{\raggedright Delft Institute of Applied Mathematics, TU Delft,   Netherlands;~\href{mailto:r.versendaal@tudelft.nl}{\texttt{r.versendaal@tudelft.nl}}.}}

\date{\today}

\begin{document}

\maketitle

\begin{abstract}
    We consider weighted geodesic random walks in a complete Riemannian manifold $(M,g)$. We show that for almost all sequences of weights (with respect to a suitable measure), these weighted geodesic random walks satisfy, when suitably scaled, a large deviation principle with a universal rate function. This extends the results from \cite{GKR16}, where this was shown for the real-valued case. It turns out the argument is also valid for general vector spaces. This allows us to use the methodology of \cite{Ver19}, in which large deviations for geodesic random walks are obtained from large deviation estimates for associated random walks in tangent spaces.  
\end{abstract}

\medskip

\noindent
{\em Key words:} Geodesic random walks, weighted random walks, large deviations, Cram\'er's theorem, stochastic processes in manifolds 

\medskip

\noindent
{\em 2020 Mathematics Subject Classification:} 60F10, 60G50, 60D05 .

\section{Introduction}

Let $X_1,X_2,\ldots$ be a sequence of independent, identically distributed random variables in $\RR$ and consider the random walk $S_n = \sum_{i=1}^n X_i$. A classical result in large deviation theory is  Cram\'er's theorem (\cite[Theorem 2.2.3]{DZ98}), which roughly states that
$$
\PP\left(\frac1nS_n \approx x\right) \approx e^{-nI(x)}
$$
where 
$$
I(x) = \sup_{t\in \RR} \left\{tx - \Lambda(t))\right\}.
$$
with $\Lambda(t) = \log\EE(e^{tX_1})$.

In \cite{GKR16}, large deviations were studied for weighted random walks. More precisely, given unit vectors $\theta^n \in \RR^n$, consider the random variables
$$
W_n^{\theta^n} := \frac1{\sqrt n}\sum_{i=1}^n \theta_i^nX_i. 
$$
It is then shown that for almost all sequences $\{\theta^n\}_n$ of weights (with respect to the measure $\sigma$ as defined in Section \ref{sec:weights_distribution}), we have
$$
\PP\left(W_n^{\theta^n} \approx x\right) \approx e^{-nI_w(x)}.
$$
Here,
$$
I_w(x) = \sup_{t\in \RR} \left\{tx - \Psi(t)\right\},
$$
where $\Psi(t) = \EE(\Lambda(tZ))$ with $Z \sim N(0,1)$. In particular, the rate function $I_w$ is independent of the weights $\theta^n$. This shows that for most weights, the associated weighted random walks have the same large deviations. Moreover, we see that the rate function in Cram\'er's theorem is different from this universal rate function $I_w$. As argued in \cite{GKR16}, this shows that Cram\'er's theorem is in a sense 'atypical'.\\ 

Our aim is to extend the result from \cite{GKR16} to weighted geodesic random walks in Riemannian manifolds. Geodesic random walks are piecewise geodesic paths, where the directions of the geodesics are chosen at random. The weights then determine the time for which we follow each direction. We refer to Section \ref{sec:GRW_weights} for a detailed description. 

In \cite{Ver19} it was shown how Cram\'er's theorem for general vector spaces can be extended to geodesic random walks. Upon analysing the proof in \cite{GKR16}, one realizes that the result remains valid in $\RR^d$, and ultimately in a general vector space $V$, see Theorem \ref{thm:GKR_LDP_vector} for a precise statement. This opens up the possibility to apply techniques from \cite{Ver19} to study large deviations for weighted geodesic random walks based on the results in \cite{GKR16}. A key step in this methodology is the splitting of the geodesic random walks in smaller pieces. Where in \cite{Ver19} these pieces are identically distributed, because of the weights this is no longer the case. We overcome this by using the symmetries of the measure $\sigma$ on the weights (see Section \ref{sec:weights_distribution}) showing that each piece of the weighted geodesic random walk still follows the same large deviations, which is ultimately what we need. 

Furthermore, our work demonstrates that the approach in \cite{Ver19} is rather robust, and emphasizes the relevant properties of the stochastic processes for the methodology to work. In particular, it motivates that the methods in \cite{Ver19} can be extended to a general framework to study large deviations for discrete-time processes in Riemannian manifolds from their Euclidean counterparts. This can for instance be used to obtain a Riemannian analogue of the Gartner-Ellis theorem (see \cite[Section 2.3]{DZ98}) and large deviations for Markov chains with values in Riemannian manifolds. Furthermore, it can be used to extend recent results on large deviations for random projections of $l^p$-balls (\cite{GKR17,APT18,LR24}) and projections on finite-dimensional subspaces (\cite{SR23}) to Riemannian manifolds. This will be the topic of future work.\\ 

The paper is structured as follows. In Section \ref{sec:random_weight_GRW_main_thm} we define weighted geodesic random walks and state our main theorem (Theorem \ref{thm:weighted_GRW_LDP}). Furthermore, we formulate the extension of large deviation for weighted random walks as in \cite{GKR16} to general vector spaces. In Section \ref{sec:proof} we prove Theorem \ref{thm:weighted_GRW_LDP}. Following the ideas from \cite{Ver19}, the proof is split in two parts, proving the upper bound and lower bound for the large deviation principle for the weighted geodesic random walks separately.

\section{Randomly weighted geodesic random walks}\label{sec:random_weight_GRW_main_thm}

Let $(M,g)$ be a complete Riemannian manifold. Denote by $d$ the Riemannian distance function. In this section we define weighted geodesic random walks in $M$, and state our main result, Theorem \ref{thm:weighted_GRW_LDP}. Furthermore, we provide the extension of the large deviation result from \cite{GKR16} to general vector spaces, which is essential for the proof of Theorem \ref{thm:weighted_GRW_LDP}. We conclude with a discussion on how the result for vector spaces relates to large deviations for $k$-dimensional projections as in \cite{SR23}.

\subsection{Geodesic random walks with weighted increments}\label{sec:GRW_weights}

In a manifold, we cannot define random walks as sums of random variables. Instead, \emph{geodesic random walks} are defined recursively by following pieces of geodesics (see e.g. \cite{KRV18,Ver19}). We then introduce the weights as the time for which we follow each piece of geodesic.
 
The procedure of following geodesics is encoded by the \emph{Riemannian exponential map}. The map $\Exp:TM \to M$ is defined by $\Exp_xv = \Exp(x,v) = \gamma(1)$, where $\gamma:[0,1] \to M$ is the geodesic with $\gamma(0) = x$ and $\dot\gamma(0) = v$. Since we assume $M$ is complete, $\Exp$ is defined on all of $TM$. With this notation at hand, we define weighted geodesic random walks.

\begin{definition}[Weighted geodesic random walks]\label{def:scaled_GRW}
Fix $x_0 \in M$, $n \in \NN$ and let $\alpha^n \in \RR^n$. A pair $(\{\Ss_k^\alpha\}_{0\leq k \leq n},\{X_k\}_{1 \leq k \leq n})$ is called a \emph{weighted geodesic random walk} with increments $\{X_k\}_{1\leq k\leq n}$ and weights $\alpha^n$, and started at $x_0$, if the following hold:
\begin{enumerate}
\item $\Ss_0^{\alpha^n} = x_0$,
\item $X_k \in T_{\Ss_{k-1}^{\alpha^n}}M$ for all $1 \leq k \leq n$,
\item $\Ss_k^{\alpha^n} = \Exp_{\Ss_{k-1}^{\alpha^n}}(\alpha^n_kX_k)$ for all $1 \leq k \leq n$.
\end{enumerate}
\end{definition}

We will consider the increments $\{X_n\}_{n\geq1}$ to be random variables. Note that the tangent space from which the next increment is drawn depends on the current position of the geodesic random walk. Therefore, we will put a collection $\{\mu_x\}_{x\in M}$ of probability measures on the tangent bundle, where $\mu_x \in \Pp(T_xM)$ is a probability measure on $T_xM$ for every $x \in M$.

\subsubsection{Identically distributed increments and parallel transport}

To compare probability distributions on different tangent spaces, we identify tangent spaces at different points using parallel transport. For $x,y \in M$ and a curve $\gamma$ connecting $x$ and $y$, we denote by $\tau_{\gamma;xy}:T_xM \to T_yM$ parallel transport along $\gamma$. If $\gamma$ is a shortest geodesic between $x$ and $y$, we simply write $\tau_{xy}$, omitting the reference to $\gamma$. Generally, we only use this notation when $x$ and $y$ are sufficiently close, so that the shortest geodesic is unique.\\

We say the measures $\mu_x$ and $\mu_y$ are \emph{identical} if for any piecewise smooth curve $\gamma$ connecting $x$ and $y$ we have
$$
\mu_x = \mu_y\circ \tau_{\gamma;xy},
$$
i.e., the distributions are invariant under parallel transport along \emph{any} piecewise smooth curve.

\subsubsection{Weight distribution}\label{sec:weights_distribution}

Our goal is to show that for almost all weights, the sequence of weighted geodesic random walks, when suitably scaled, satisfies a large deviation principle with a universal rate function. For such statements to make sense, we need to introduce a probability distribution on the space of weights. For this, we follow \cite{GKR16}.

Let $\bS^{n-1}$ be the unit sphere in $\RR^n$ and denote by $\sigma_{n-1}$ the uniform measure on $\bS^{n-1}$. Define the product space $\bS = \prod_{n=1}^\infty \bS^{n-1}$ with projections $\pi_n:\bS \to \bS^{n-1}$. We consider a probability measure $\sigma$ on $\bS$ such that
$$
\sigma \circ \pi_n = \sigma_{n-1}
$$
for all $n \in \NN$.

\subsection{Main result}

Our main result extends Theorem 2 from \cite{GKR16} to weighted geodesic random walks. For this, consider $\theta \in \bS$ a triangular array of weights. We define $\{\Ss_k^{\theta^n}\}_{0\leq k\leq n}$ to be the weighted geodesic random walk with respect to the scaled weights $\alpha^n = \frac{1}{\sqrt n}\theta^n$ (see Definition \ref{def:scaled_GRW}). Our main result concerns the large deviations for the sequence $\{\Ss_n^{\theta^n}\}_{n\in\NN}$.

\begin{theorem}\label{thm:weighted_GRW_LDP}
    Let $\theta \in \bS$ and let $\{\Ss_n^{\theta^n}\}_{n \in \NN}$ be the weighted geodesic random walk as defined above. Assume the increments of the geodesic random walks are bounded, independent and identically distributed. Let $\sigma$ be a measure on $\bS$ as in Section \ref{sec:GRW_weights}. Then for $\sigma$-almost every $\theta$, the sequence $\{\Ss_n^{\theta^n}\}_{n \in \NN}$ satisfies a large deviation principle with rate function
    $$
    I(x) = \inf_{v \in \exp_{x_0}^{-1}x} \sup_{\lambda\in T_{x_0}M} \ip{\lambda}{v} - \Psi(\lambda),
    $$
    where $\Psi(\lambda) = \EE(\Lambda_{x_0}(Z\lambda))$ with $Z \sim N(0,1)$ and $\Lambda_{x_0}(\lambda) = \log\int_{T_{x_0}M} e^{\ip{\lambda}{w}}\mu_{x_0}(\dd w)$.
\end{theorem}

The proof is inspired by the proof of Cram\'er's theorem for geodesic random walks in \cite{Ver19}. The key idea is to associate the weighted geodesic random walk $\Ss_n^{\theta^n}$ to a weighted random walk in $T_{x_0}M$. The large deviations for this associated random walk follow from Theorem \ref{thm:GKR_LDP_vector}. Unfortunately, the connection between the two random walks does not immediately allow us to tranfer the large deviations to $\Ss_n^{\theta^n}$. Instead, we carefully analyse the connection between the two processes and prove the upper bound and lower bound of the large deviation principle for $\{\Ss_n^{\theta^n}\}_{n \in \NN}$ separately. In particular, Theorem \ref{thm:weighted_GRW_LDP} follows immediately from Propositions \ref{prop:upper_bound_LDP_theorem} and \ref{prop:lower_bound_LDP_theorem}.

\subsubsection{Large deviations for weighted random  walks in vector spaces}

To prove Theorem \ref{thm:weighted_GRW_LDP}, we need an extension of \cite[Theorem 2]{GKR16} to and arbitrary vector space $V$, which for us will be the tangent space $T_{x_0}M$. Let $X_1,X_2,\ldots$ be a sequence of independent, identically distributed random variables in $V$. Let $\theta \in \bS$ be a sequence of coefficients. We define
\begin{equation}\label{eq:weighted_RW_vector_V}
W_k^{\theta^n} := \frac{1}{\sqrt n}\sum_{i=1}^k \theta_i^nX_i
\end{equation}
for $k = 1,2,\ldots,n$.

The proof in \cite{GKR16} for the large deviations for weighted random walks in $\RR$ extends to higher dimensional Euclidean spaces. Upon choosing a basis, we get the following. 

\begin{theorem}\label{thm:GKR_LDP_vector}
    Let $X_1,X_2,\ldots$ be a sequence of independent, identically distributed random variables in a vector space $V$. Denote by $\Lambda(\lambda) = \log\EE(e^{\ip{\lambda}{X_1}})$ the log-moment generating function of $X_1$. Define the random variables $W_k^{\theta^n}$ as in \eqref{eq:weighted_RW_vector_V} and let $\sigma$ be as in Section \ref{sec:weights_distribution}. Then for $\sigma$-almost every $\theta$, the sequence $\{W_n^{\theta^n}\}_{n\in\NN}$ satisfies the large deviation principle with good rate function
    \begin{equation}\label{eq:rf_GKR}
    I(v) = \Psi^*(v) = \sup_{\lambda \in V} \left\{\ip{v}{\lambda} - \Psi(\lambda)\right\},
    \end{equation}
    where $\Psi(\lambda) = \EE_Z(\Lambda(Z\lambda))$ with $Z \sim N(0,1)$.
\end{theorem}

By using Varadhan's lemma (see e.g \cite[Theorem 4.3.1]{DZ98}), and the symmetry properties of the measure $\sigma$, we can deduce the following from Theorem \ref{thm:GKR_LDP_vector}. This is essential for our proof of Theorem \ref{thm:weighted_GRW_LDP}, where we split up the weighted geodesic random walks in smaller pieces.

\begin{corollary}\label{cor:Varadhan_limit_mgf}
    Let the assumptions of Theorem \ref{thm:GKR_LDP_vector} be satisfied. Then for $\sigma$-almost every $\theta$ we have
    $$
    \lim_{n \to \infty} \frac 1n \log \EE\left(e^{n\ip{\lambda}{W_{\lfloor k^{-1}n\rfloor}^{\theta}}}\right) = \frac1k\Psi(\lambda).
    $$ 
\end{corollary}
\begin{proof}
    By Varadhan's lemma (\cite[Theorem 4.3.1]{DZ98}) it follows from the large deviation principle in Theorem \ref{thm:GKR_LDP_vector} that
    $$
    \lim_{n\to\infty} \frac1n\log \EE\left(e^{n\ip{\lambda}{W_n^{\theta^n}}}\right) = \Psi(\lambda),
    $$
    By comparing sequences $\theta$ and $\tilde \theta$ for which $\theta^n$ and $\tilde\theta^n$ only differ in the first $\lfloor nk^{-1}\rfloor$ elements, the above implies that
    $$
    \Psi_k(\lambda) := \lim_{n\to\infty} \frac1n\log \EE\left(e^{n\ip{\lambda}{W_{\lfloor nk^{-1}\rfloor}^{\theta^n}}}\right)
    $$
    is independent of $\theta$ for $\sigma$-almost every $\theta$. 
    
    Now write $n_i = i\lfloor nk^{-1}\rfloor$ for $i = 1,\ldots,k-1$ and $n_k = n$. Since the uniform distribution on $\bS^{n-1}$ is permutation invariant, it follows that for $\sigma$-almost every $\theta$ we have
    $$
    \lim_{n\to\infty} \frac1n\log \EE\left(e^{n\ip{\lambda}{W_{n_{i+1}}^{\theta^n} - W_{n_i}^{\theta^n} }}\right) = \Psi_k(\lambda).
    $$
    for all $i = 1,\ldots,k-1$. Since also
    $$
    \lim_{n\to\infty} \frac1n\log \EE\left(e^{n\ip{\lambda}{W_n^{\theta^n}}}\right) = \sum_{i=0}^{k-1} \lim_{n\to\infty} \frac1n\log \EE\left(e^{n\ip{\lambda}{W_{n_{i+1}}^{\theta^n} - W_{n_i}^{\theta^n} }}\right),
    $$
    by independence of the increments, it follows that $\Psi_k(\lambda) = \frac1k\Psi(\lambda)$ as desired.
\end{proof}

\subsubsection{Connection to large deviations for $k$-dimensional projections}

    In \cite{SR23}, large deviations are studied for random multidimensional projections. In particular, for a sequence $Y_1,Y_2,\ldots$ of independent, identically distributed random variables in $\RR$, the authors consider (among many other things) sequences of $k$-dimensional projections of $Y^{(n)} := (Y_1,\ldots,Y_n)$. More precisely, for a sequence $\textbf{a} = \{\textbf{a}_{n,k}\}_{n\geq k}$ with each $\textbf{a}_{n,k}$ an orthonormal $k$-frame for $\RR^n$, large deviations are considered for the sequence
    $$
    P_n := \frac1{\sqrt n}\textbf{a}_{n,k}^TY^{(n)} \in \RR^k.
    $$
    It is shown that for almost all such sequences of projections (with respect to the Haar measure on the Stiefel manifold of orthogonal $k$-frames on $\RR^n$), the sequence $\{P_n\}_{n\geq k}$ satisfies the large deviation principle with rate function
    \begin{equation}\label{eq:rf_projections}
        I_{\mathrm{proj}}(v) = \sup_{\lambda \in \RR^k} \left\{\ip{\lambda}{v} - \EE_{Z_1,\ldots,Z_k}(\Lambda(\lambda_1Z_1 + \cdots + \lambda_kZ_k))\right\},
    \end{equation}
    where $Z_1,\ldots,Z_k$ are independent, standard normal random variables and $\Lambda$ is the log-moment generating function of $Y_1$.\\

    For a sequence $X_1,X_2,\ldots$ of independent, identically distributed random variables in $\RR^k$, the random variables $W_n^{\theta^n}$ as defined in \eqref{eq:weighted_RW_vector_V} can also be interpreted as such $k$-dimensional projections. Indeed, one projects the sequence $(X_1^1,\ldots,X_1^k,\ldots,X_n^1,\ldots,X_n^k) \in \RR^{nk}$ onto $\RR^k$. Therefore, one may wonder to what extent the results from \cite{SR23} are connected to the statement in Theorem \ref{thm:GKR_LDP_vector}. There are two main caveats:
    \begin{enumerate}
        \item For the results of \cite{SR23} to apply, the coordinates of $X_i$ must independent and identically distributed, i.e., the distribution $\mu$ of $X_i$ has to be of the form $\mu_1^{\otimes k}$.
        \item The projections we consider map the variables $X_1^i,X_2^i,\ldots,X_n^i$ onto the $i$-th coordinate. With respect to the measure considered in \cite{SR23}, the collection of such projections has measure 0. Additionally, for such projections, the coordinates are independent, identically distributed. 
    \end{enumerate}

    Because the set of projections we consider has measure 0, it is inconclusive whether the associated $k$-projections $P_n$ satisfy the large deviation principle with universal rate function $I_{\mathrm{proj}}$. 

    To investigate this, note that we are comparing projections of the form $P_{nk}$ to randomly weighted sums of the form $W_n^{\theta^n}$. The scaling of these random variables is different by a factor $\sqrt k$. Furthermore, the large deviations of $P_{nk}$ are at rate $nk$, while those for $W_n^{\theta^n}$ are at rate $n$. If we denote by $I_k$ the universal rate function in \eqref{eq:rf_GKR} for $\{W_n^{\theta^n}\}_{n\in \NN}$, the question thus becomes whether we have
    $$
    I_k(v) = kI_{\mathrm{proj}}\left(\frac{v}{\sqrt k}\right).
    $$
    We first consider two examples.

    \begin{example}\label{ex:normal}
        Suppose $\mu_1$ is a standard normal distribution, so that $\Lambda_{\mu_1}(t) = \frac12t^2$. Since $\mu = \mu_1^{\otimes k}$ we have
        $$
        I_k(v) = \sup_{\lambda \in \RR^k} \left\{\ip{\lambda}{v} - \sum_{i=1}^k \EE_Z(\Lambda_{\mu_1}(\lambda_iZ))\right\}
        = 
        \sup_{\lambda \in \RR^k} \left\{\ip{\lambda}{v} - \sum_{i=1}^k \frac12\lambda_i^2\right\}
        =
        \sum_{i=1}^k \frac12v_i^2 = \frac12|v|^2.
        $$
        
        Likewise, we find that
        $$
        I_{\mathrm{proj}}(v) 
        = 
        \sup_{\lambda \in \RR^k} \left\{\ip{\lambda}{v} - \frac12\EE_{Z_1,\ldots,Z_k}((\lambda_1Z_1 + \cdots + \lambda_kZ_k)^2)\right\} =
        \sup_{\lambda \in \RR^k} \left\{\ip{\lambda}{v} - \frac12|\lambda|^2\right\} = \frac12|v|^2
        $$
        Here we used that $Z_1,\ldots,Z_k$ are independent and $\EE(Z_i^2) =  1$. This shows that $I_k(v) = kI_{\mathrm{proj}}\left(\frac{v}{\sqrt k}\right)$. Note that this relies on the fact that $I_{\mathrm{proj}}(v)$ is not affected by the specific scaling with $k$.
    \end{example}

    Example \ref{ex:normal} remains true when $\mu_1$ is a general normal distribution. Intuitively, this can be explained as follows. Take any orthonormal $k$-frame $\mathbf{a}$ of $\RR^n$ and let $X^{(n)} = (X_1,X_2,\ldots,X_n)$ be a vector of independent, identically distributed random variables with a normal distribution. Let us write $a_1,\ldots,a_n$ for the rows of $\mathbf a$. Then
    $$
    \mathbf{a}^TX^{(n)} = \sum_{l=1}^n X_la_l \in \RR^k.
    $$ 
    In particular, this shows that each coordinate of the projection has a normal distribution. Furthermore, we can compute
    $$
    \Cov\left(\sum_{l=1}^n X_la_l^i,\sum_{l=1}^n X_la_l^j\right) = \EE(X_1^2)\sum_{l=1}^n a_l^ia_l^j = 0
    $$
    since the columns of $\mathbf a$ are orthogonal. This shows the coordinates of the projection are uncorrelated, and hence independent since they are normally distributed. This means we can treat the coordinates separately, each one being a randomly weighted sum as in \cite{GKR16}. Since this also holds for our special projections, the result of Example \ref{ex:normal} is indeed expected.
    
    \begin{example}\label{ex:poisson}
        Take $k = 2$ and $\mu_1$ to be Poi$(1)$. A computation shows that
        $$
        I_2(v) = \sup_{\lambda \in \RR^2} \left\{\lambda_1v_1 + \lambda_2v_2 - e^{\frac12\lambda_1^2} - e^{\frac12\lambda_2^2} + 2\right\}
        $$
        and
        $$
        I_{\mathrm{proj}}(v) = \sup_{\lambda \in \RR^2} \left\{\lambda_1v_1 + \lambda_2v_2 - e^{\frac12|\lambda|^2} + 1\right\}
        $$
        Numerically, we obtain that
        $$
        I_2((1,2)) \approx 1.7940, \qquad 2I_{\mathrm{proj}}\left(\frac{1}{\sqrt 2}(1,2)\right) \approx 1.8662.
        $$
        We thus find $v$ for which $I_k(v) \neq kI_{\mathrm{proj}}\left(\frac{v}{\sqrt k}\right)$. 
    \end{example}

    As Example \ref{ex:poisson} demonstrates, large deviation principles associated to the special sequences of $k$-projections we consider in general have a rate function different from the universal coming from \cite{SR23}. The reason is that for our special $k$-projections, the coordinates are independent and have the same distribution. This is in contrast to a typical sequence of $k$-projections, in which all elements are mapped to any of the coordinates, making their joint distribution differ significantly from a product distribution.

    However, as mentioned above, for our special projections, each coordinate is independent and has the same distribution. Moreover, we can interpret coordinate $i$ as a 1-projection of the sequence $(X_1^i,X_2^i,\ldots,X_n^i) \in \RR^n$. This motivates the following result.

    \begin{proposition}
        Let $\mu_1$ be a probability measure on $\RR$ and set $\mu = \mu_1^{\otimes k}$. Denote by $\Lambda_{\mu},\Lambda_{\mu_1}$ the log-moment generating function of the measure $\mu$, respectively $\mu_1$. Define
        $$
        I_k(v) = \sup_{\lambda \in \RR^k} \left\{\ip{v}{\lambda} - \EE_Z(\Lambda_\mu(Z\lambda))\right\}
        $$ 
        and
        $$
        I_{\mathrm{proj}}(v) = \sup_{\lambda \in \RR^k} \left\{\ip{v}{\lambda} - \EE_{Z_1,\ldots,Z_k}(\Lambda_\mu(\lambda_1Z_1 + \cdots + \lambda_kZ_k)\right\},
        $$
        where $Z,Z_1,\ldots,Z_k$ are independent, standard normal random variables. Then for all $c \in \RR$ we have
        $$
        I_k(c\One_k) = kI_{\mathrm{proj}}\left(\frac{c}{\sqrt k}\One_k\right),
        $$
        where $\One_k \in \RR^k$ denotes the all-ones vector.
    \end{proposition}
    \begin{proof}
        By independence we have
        $$
        I_k(c\One_k) = \sum_{i=1}^k \sup_{\lambda_i \in \RR} \left\{c\lambda_i - \EE_Z(\Lambda_{\mu_1}(\lambda_i Z))\right\} = k\sup_{\lambda \in \RR} \left\{c\lambda - \EE_Z(\Lambda_{\mu_1}(\lambda Z))\right\}
        $$
        On the other hand, if $Z_1,\ldots,Z_k$ are independent, standard normal random variables, then for $\lambda \in \RR^k$ we have $\ip{\lambda}{(Z_1,\ldots,Z_k)} \stackrel{d}= |\lambda|Z$ with $Z$ standard normal. This implies that
        $$
        I_{\mathrm{proj}}(v) = \sup_{\lambda \in \RR^k} \left\{\ip{\lambda}{v} - \EE_Z(\Lambda_{\mu_1}(|\lambda|Z))\right\}
        $$
        For $v = c\One_k$, the optimal $\lambda$ will be of the form $t\One_k$. As a consequence, we get
        $$
        I_{\mathrm{proj}}(c\One_k) = \sup_{t \in \RR} \left\{ktc - \EE_Z(\Lambda_{\mu_1}(\sqrt k|t|Z))\right\} = \sup_{t \in \RR} \left\{ktc - \EE_Z(\Lambda_{\mu_1}(\sqrt ktZ))\right\},
        $$
        where we used that $|t|Z \stackrel{d}= tZ$. From this it follows that
        $$
        kI_{\mathrm{proj}}\left(\frac{c}{\sqrt k}\One_k\right) = k\sup_{t \in \RR} \left\{\sqrt ktc - \EE_Z(\Lambda_{\mu_1}(\sqrt ktZ))\right\} = k\sup_{t \in \RR} \left\{tc - \EE_Z(\Lambda_{\mu_1}(tZ))\right\} = I_k(c\One_k) 
        $$
        as desired.
    \end{proof}

\section{Proof of Theorem \ref{thm:weighted_GRW_LDP}} \label{sec:proof}

As explained below Theorem \ref{thm:weighted_GRW_LDP}, we prove the lower and upper bound of the large deviation principle separately. In particular, Theorem \ref{thm:weighted_GRW_LDP} follows immediately from Propositions \ref{prop:lower_bound_LDP_theorem} and \ref{prop:upper_bound_LDP_theorem}. Before we get to those, we first introduce some notation and preliminary results.

\subsection{Notation and preliminary results}

Our proof is inspired by the methods in \cite{Ver19}. Therefore, we rely on connecting the random walks $\{\Ss_l^{\theta^n}\}_{0\leq l \leq n}$ to weighted random walks in the tangent space $T_{x_0}M$. Naively, we could consider
$$
v_l^n := \exp_{x_0}^{-1}\left(\Ss_l^{\theta^n}\right) \in T_{x_0}M.
$$
However, two problems arise. Most importantly, the Riemannian exponential map need not be invertible, so that $v_l^n$ cannot be uniquely defined. But even if it is invertible, it turns out that curvature forms an obstruction to compare $v_n^n$ to a random walk with increments distributed according to $\mu_{x_0}$. Indeed, one expects to compare $v_n^n$ to the weighted random walk
$$
\frac{1}{\sqrt n}\sum_{l=1}^n \theta_l^n\tau_{x_0 \Ss_{l-1}^{\theta^n}}^{-1}X_l^n \in T_{x_0}M.
$$
Because the distributions $\mu_x$ of the increments are invariant under parallel transport, we indeed  have $\tau_{x_0 \Ss_{l-1}^{\theta^n}}^{-1}X_l^n \sim \mu_{x_0}$. Therefore, by Theorem \ref{thm:GKR_LDP_vector}, these random walks in $T_{x_0}M$ satisfy for $\sigma$-almost every $\theta$ a large deviation principle with rate function $\Psi_{x_0}^*$. One then hopes to use a contraction principle to also obtain a large deviations for $\{\Ss_n^{\theta^n}\}_n$. Unfortunately, from Proposition \ref{prop:RW_to_tangent_space} (by taking $l = n$) we see that the difference between $v_n^n$ and the proposed random walk in $T_{x_0}M$ is $\Oo(1)$. The connection to a weighted random walks in the tangent space therefore needs to be refined.\\

We solve both issues by introducing a parameter $m \in \NN$ and partitioning the random walk in $m$ parts, each consisting of (roughly) $\lfloor nm^{-1}\rfloor$ steps. More precisely, define $n_i = i\lfloor nm^{-1}\rfloor$ for $i = 0,1\ldots,m-1$ and set $n_m = n$. We have the following bound on how far our random walks can wander in $\lfloor nm^{-1}\rfloor$ steps.

\begin{lemma}\label{lem:distance_RW}
    Fix $\theta^n \in \bS^{n-1}$ and let $r$ be the uniform upper bound on the increments of $\Ss_k^{\theta^n}$. Then
    $$
    d\left(\Ss_{k}^{\theta^n},x_0\right) \leq \frac{r\sqrt k}{\sqrt n} \leq r
    $$
    for $k = 1,\ldots,n$. Similarly, if $n_i = i\lfloor nm^{-1}\rfloor$, we have
    $$
    d\left(\Ss_{n_i}^{\theta^n},\Ss_{n_i+k}^{\theta^n}\right) \leq \frac{r}{\sqrt m}
    $$
    for $1 \leq k \leq \lfloor nm^{-1}\rfloor$.
\end{lemma}
\begin{proof}
    By the Cauchy-Schwarz inequality and the fact that $\theta^n \in \bS^{n-1}$, we have
    \begin{equation}\label{eq:CS_coef}
    \sum_{i=1}^{k} |\theta_i^n| \leq \sqrt k\sum_{i=1}^{k} |\theta_i^n|^2 \leq \sqrt k.
    \end{equation}
    The triangle inequality then gives us
    $$
    d\left(\Ss_{k}^{\theta^n},x_0\right) \leq \frac{1}{\sqrt n}\sum_{i=1}^{k} |\theta_i^n||X_i^n| \leq \frac{r\sqrt k}{\sqrt n}. 
    $$
    The second claim follows similarly by using that $k \leq \frac{n}{m}$.
\end{proof}

Lemma \ref{lem:distance_RW} shows that the distance between points in the same piece of length $\lfloor nm^{-1}\rfloor$ of the random walk decays with $m$. This allows us to refine the estimate in Proposition \ref{prop:RW_to_tangent_space}, and let $m$ tend to infinity in the end. Furthermore, Lemma \ref{lem:distance_RW} also resolves the issue of the non-invertibility of the Riemannian exponential map. To see this, for $x \in M$ we first define the \emph{injectivity radius}
$$
\iota(x) = \sup\{t > 0| \exp_x \mbox{ is injective on } B(0,t)\}.
$$ 
It turns out that $\iota$ is continuous on $M$ (see e.g. \cite{Kli82}). By Lemma \ref{lem:distance_RW}, the random walks $\{\Ss_l^{\theta^n}\}_{0\leq l \leq n}$ all remain in the set $K = \overline{B(x_0,r)}$, which is compact because $M$ is complete. As a consequence, $\iota$ attains a minimum $\iota_K > 0$ on $K$, meaning that $\exp_x$ is injective on $B(0,\iota_K) \subset T_xM$ for all $x \in K$. In particular, if we take $m$ large enough so that $\frac{r}{\sqrt m} < \iota_K$, by Lemma \ref{lem:distance_RW} we can uniquely define the random vectors
\begin{equation}\label{eq:pullback_vector_piece}
\tilde v_k^{n,m,i} \in \Exp_{\Ss_{n_{i-1}}^{\theta^n}}^{-1}\left(\Ss_{n_{i-1}+k}^{\theta^n}\right) \subset T_{\Ss_{n_{i-1}}^{\theta^n}}M
\end{equation}
of minimal length.

The following result shows how we can compare the random variables $v_l^n := \tilde v_l^{n,m,1}$ (which is independent of $m$, since $\Ss_{n_0} = x_0$) to random variables of the form
$$
\frac{1}{\sqrt n}\sum_{k=1}^l \theta_k^n\tau_{x_0 \Ss_{k-1}^{\theta^n}}^{-1}X_k^n.
$$
As discussed above, the latter is a random walk in $T_{x_0}M$ with increments distributed as $\mu_{x_0}$.

\begin{proposition}\label{prop:RW_to_tangent_space}
    Fix $n \in \NN$, $\theta^n \in \bS^{n-1}$ and let $r > 0$ be such that $|X_l^n| \leq r$ for all $1 \leq l \leq n$. Then there is a constant $C > 0$ such that
    $$
    \left|v_l^n - \frac{1}{\sqrt n}\sum_{k=1}^l \theta_k^n\tau_{x_0 \Ss_{k-1}^{\theta^n}}^{-1}X_k^n\right| \leq C\frac{1}{n}\sum_{k=1}^l (\theta_k^n)^2 + Cr^3\frac{l^{3/2}}{n^{3/2}}.
    $$
    for all $1 \leq l \leq n$ for which $\dd(\Ss_k^{\theta^n},x_0) < \iota(x_0)$ for all $1 \leq k \leq l-1$.
\end{proposition}
\begin{proof}
    From Lemma \ref{lem:distance_RW} we know that $\Ss_k^{\theta^n}$ remains inside the compact set $\overline{B(x_0,r)}$. Therefore, by \cite[Proposition 5.4]{Ver19}, there exists a constant $C > 0$ such that
    \begin{equation}\label{eq:induction_step}
    \left|v_{l+1}^n - \left(v_l^n + \frac{\theta_l^n}{\sqrt n}\dd(\Exp_{x_0})^{-1}_{v_l^n}X_{k+1}^n\right)\right| \leq  C\frac{(\theta_l^n)^2}{n}.
    \end{equation}

    To use this to prove the desired statement, we first apply the triangle inequality to obtain
    \begin{multline}
    \left|v_l^n - \frac1{\sqrt n}\sum_{k=1}^l \theta_k^n\tau_{x_0\Ss_{k-1}^{\theta^n}}^{-1}X_k^n\right| 
    \\
    \leq
    \left|v_l^n - \frac1{\sqrt n}\sum_{k=1}^l \theta_k^n\dd(\Exp_{x_0})_{v_{k-1}^n}^{-1}X_k^n\right| + \frac1{\sqrt n}\sum_{k=1}^l|\theta_k^n|\left|\dd(\Exp_{x_0})_{v_{k-1}^n}^{-1}X_k^n -  \tau_{x_0\Ss_{k-1}^{\theta^n}}^{-1}X_k^n\right|.
    \end{multline}

    We estimate both terms separately. By telescoping, we can use \eqref{eq:induction_step} to estimate the first term:
    \begin{equation} \label{eq:sum_differential}
    \begin{aligned} 
    \left|v_l^n - \frac1{\sqrt n}\sum_{k=1}^l \theta_k^n\dd(\Exp_{x_0})_{v_{k-1}^n}^{-1}X_k^n\right| 
    &\leq 
    \sum_{k=1}^l |v_k^n - v_{k-1}^n -  \theta_k^n\dd(\Exp_{x_0})_{v_{k-1}^n}^{-1}X_k^n| 
    \\
    &\leq
    C\frac 1n\sum_{k=1}^l (\theta_k^n)^2. 
    \end{aligned}
    \end{equation}

    For the second term, we apply \cite[Corollary 5.8]{Ver19} to obtain
    $$
    \frac1{\sqrt n}\sum_{k=1}^l|\theta_k^n|\left|\dd(\Exp_{x_0})_{v_{k-1}^n}^{-1}X_k^n -  \tau_{x_0\Ss_{k-1}^{\theta,n}}^{-1}X_k^n\right| 
    \leq 
    C\frac r{\sqrt n}\sum_{k=1}^l |\theta_k^n||v_{k-1}^n|^2
    \leq
    C\frac {r^3}{n^{3/2}}l\sum_{k=1}^l |\theta_k^n|
    \leq
    Cr^3\frac{l^{3/2}}{n^{3/2}}.
    $$
    Here we used that $|v_{k-1}^n| = d\left(\Ss_{k-1}^{\theta^n},x_0\right)$ together with Lemma \ref{lem:distance_RW} and \eqref{eq:CS_coef}. 
  
\end{proof}

\subsection{The upper bound of the large deviation principle}

To prove large deviation bounds for $\{\Ss_n^{\theta^n}\}_{n \in \NN}$, we first consider the random variables $\tilde v_{\lfloor m^{-1}n \rfloor}^{n,m,k}$ as defined in \eqref{eq:pullback_vector_piece}. However, these random variables still live in different tangent spaces, which depend on the trajectory of the random walk. Therefore, we first transport all these random tangent vectors back to $T_{x_0}M$.\\ 

For the upper bound of the large deviation principle, the only relevant property is that the transport of $\tilde v_{\lfloor m^{-1}n \rfloor}^{n,m,i}$ is measurable with respect to $\sigma (\Ss_l^{\theta^n}, 0 \leq l \leq n_{i-1})$. Therefore, we choose to carry out the parallel transport via the intermediate points $\Ss_{n_1}^{\theta^n},\ldots,\Ss_{n_{i-1}}^{\theta^n}$ of the random walk. For $m$ large enough, consecutive points of this form can be connected with a unique shortest geodesic, and we denote by $\tau_{\Ss_{n_{j-1}}^{\theta^n}\Ss_{n_j}^{\theta^n}}$ parallel transport along this geodesic. Using this notation, we define parallel transport $\tau_{RW,i}:T_{x_0}M \to T_{\Ss_{n_i}^{\theta^n}}M$ by 
\begin{equation}\label{eq:parallel_transport_RW}
\tau_{RW,i} = \tau_{\Ss_{n_{i-1}}^{\theta^n}\Ss_{n_i}^{\theta^n}}\circ\tau_{\Ss_{n_{i-2}}^{\theta^n}\Ss_{n_{i-1}}^{\theta^n}}\circ\cdots\circ\tau_{x_0\Ss_{n_1}^{\theta^n}}
\end{equation}
We then define
\begin{equation}\label{eq:tangent_RW_base_upper}
    v_{\lfloor m^{-1}n\rfloor}^{n,m,i} = \tau_{RW,i-1}^{-1}\tilde v_{\lfloor m^{-1}n\rfloor}^{n,m,i}.
\end{equation}

The aim is now to derive the upper bound of the large deviation principle for $\{\Ss_n^{\theta^n}\}_{n \in \NN}$ by studying the large deviations upper bound for the random variables
\begin{equation}
\left(v_{\lfloor m^{-1}n \rfloor}^{n,m,1},\ldots,v_{\lfloor m^{-1}n \rfloor}^{n,m,m}\right) \in (T_{x_0}M)^m. 
\end{equation}
With \cite[Theorem 4.5.3]{DZ98} in mind, we first need to understand
$$
\limsup_{n\to\infty} \frac1n\log\EE\left(e^{n\sum_{i=1}^m \inp{\lambda_i}{v_{\lfloor m^{-1}n\rfloor}^{n,m,i}}}\right).
$$
Unfortunately, it cannot be computed exactly. Instead, we compare $v_{\lfloor m^{-1}n\rfloor}^{n,m,i}$ to sums of random variables given by
\begin{equation}\label{eq:Y_upper}
Y_i^n = \tau_{RW,i-1}^{-1}\sum_{k=n_{i-1} + 1}^{n_i} \theta_k^n\tau_{\Ss_{n_{i-1}}^{\theta^n}\Ss_{k-1}^{\theta^n}}^{-1}X_k^n \in T_{x_0}M.
\end{equation}
Computing the moment generating function of sums of independent, identically distributed random variables is straightforward. However, the weights $\theta^n$ make the sums in $Y_i^n$ inhomogeneous. Therefore, contrary to the work on Cram\'er's theorem for geodesic random walks in \cite{Ver19}, we now have to use of Varadhan's Lemma, more precisely Corollary \ref{cor:Varadhan_limit_mgf}, to obtain asymptotics of the moment generating functions of the $Y_i^n$.

\begin{proposition}\label{prop:upperbound_MGF}
    Let the assumptions of Theorem \ref{thm:weighted_GRW_LDP} be satisfied. Denote by $r$ the uniform bound on the increments of the geodesic random walk. Consider the random variables $\left(v_{\lfloor m^{-1}n \rfloor}^{n,m,1},\ldots,v_{\lfloor m^{-1}n \rfloor}^{n,m,m}\right)$ defined in \eqref{eq:tangent_RW_base_upper}. Then there exists a constant $C > 0$ such that for $m$ large enough and all $(\lambda_1,\ldots,\lambda_k)\in (T_{x_0}M)^m$ we have
    $$
    \limsup_{n\to\infty} \frac1n\log\EE\left(e^{n\sum_{i=1}^m \inp{\lambda_i}{v_{\lfloor m^{-1}n\rfloor}^{n,m,i}}}\right) \leq \frac1m\sum_{i=1}^m \Psi_{x_0}(\lambda_i) + \frac{Cr^3}{m^{3/2}}\sum_{i=1}^m |\lambda_i|
    $$
    for $\sigma$-almost every $\theta$.
\end{proposition}

\begin{proof}
    Consider the random variables $Y_i^n$ as defined in \eqref{eq:Y_upper}. By the Cauchy-Schwarz and triangle inequality, and the fact that parallel transport is an isometry, we have
    $$
    \left|\sum_{i=1}^m \inp{\lambda_i}{v_{\lfloor m^{-1}n\rfloor}^{n,m,i}} - \frac1{\sqrt n}\sum_{i=1}^m \inp{\lambda_i}{Y_i^n}\right| 
    \leq \sum_{i=1}^m |\lambda_i|\left|\tilde v_{\lfloor m^{-1}n\rfloor}^{n,m,i} - \frac1{\sqrt n}\sum_{k=n_{i-1}+1}^{n_i} \theta_k^n\tau_{\Ss_{n_{i-1}}^{\theta^n}\Ss_k^{\theta^n}}^{-1}X_k^n\right|.
    $$

    We further estimate this using Proposition \ref{prop:RW_to_tangent_space} to obtain
    \begin{align*}
    \left|\sum_{i=1}^m \inp{\lambda_i}{v_{\lfloor m^{-1}n\rfloor}^{n,m,i}} - \frac1{\sqrt n}\sum_{i=1}^m \inp{\lambda_i}{Y_i^n}\right| 
    &\leq 
    C\sum_{i=1}^m \left(\frac{1}{n}|\lambda_i|\sum_{k=n_{i-1}+1}^{n_i} (\theta_k^n)^2 + |\lambda_i|r^3\frac{(n_i - n_{i-1})^{3/2}}{n^{3/2}}\right)
    \\
    &\leq
    C\max_i|\lambda_i|\frac1n\sum_{i=1}^m \sum_{k=n_{i-1}+1}^{n_i} (\theta_k^n)^2 + \frac{Cr^3}{m^{3/2}}\sum_{i=1}^m |\lambda_i|
    \\
    &=
    C\max_i|\lambda_i|\frac1n + \frac{Cr^3}{m^{3/2}}\sum_{i=1}^m |\lambda_i|.
    \end{align*}
    Here we used that $n_i - n_{i-1} = \lfloor nm^{-1}\rfloor \leq nm^{-1}$ and
    $$
    \sum_{i=1}^m \sum_{k=n_{i-1}+1}^{n_i} (\theta_k^n)^2 = \sum_{k=1}^n (\theta_k^n)^2 = 1,
    $$
    because $\theta^n \in \bS^{n-1}$.\\
    
    Applying this estimate to the moment generating function, we obtain
    $$
    \EE\left(e^{n\sum_{i=1}^m\inp{\lambda_i}{v_{\lfloor m^{-1}n\rfloor}^{n,m,i}}}\right)
    \leq 
    e^{C\max_i|\lambda_i|}e^{Cr^3nm^{-3/2}\sum_{i=1}^m|\lambda_i|}\EE\left(e^{\sqrt n\sum_{i=1}^m \inp{\lambda_i}{Y_i^n}}\right).
    $$
    It follows that
    \begin{align*}
    \limsup_{n\to\infty} \frac1n\log\EE\left(e^{n\sum_{i=1}^m \inp{\lambda}{v_{\lfloor m^{-1}n\rfloor}^{n,m,i}}}\right) 
    &\leq 
    \limsup_{n\to\infty} \frac{C\max_i|\lambda_i|}{n} + \frac{Cr^3}{m^{3/2}}\sum_{i=1}^m |\lambda_i| + \frac1n\log \EE\left(e^{\sqrt n\sum_{i=1}^m \inp{\lambda_i}{Y_i^n}}\right)
    \\
    &=
    \frac{Cr^3}{m^{3/2}}\sum_{i=1}^m |\lambda_i| + \limsup_{n\to\infty}\frac1n\log \EE\left(e^{\sqrt n\sum_{i=1}^m \inp{\lambda_i}{Y_i^n}}\right)
    \end{align*}

    We now compute the asymptotics for the remaining moment generating function. Because the increments of the geodesic random walk are parallel transport invariant and independent, the random variables $\tau_{RW,i-1}^{-1}\tau_{\Ss_{n_{i-1}}^{\theta^n}\Ss_{k-1}^{\theta^n}}^{-1}X_k^n$ are independent with distribution $\mu_{x_0}$. This implies that $Y_1^n,\ldots,Y_m^n$ are independent, and hence 
    $$
    \EE\left(e^{\sqrt n\sum_{i=1}^m \inp{\lambda_i}{Y_i^n}}\right) = \prod_{i=1}^m \EE\left(e^{\sqrt n \inp{\lambda_i}{Y_i^n}}\right).
    $$
    Moreover, by applying Corollary \ref{cor:Varadhan_limit_mgf} to the random variables $\frac{1}{\sqrt n}Y_i^n$, we find that for $\sigma$-almost every $\theta$ and all $i = 1,\ldots,m$ we have
    $$
    \lim_{n\to\infty} \frac1n\log\EE\left(e^{\sqrt n \inp{\lambda_i}{Y_i^n}}\right) = \frac{1}{m}\Psi_{x_0}(\lambda_i),
    $$
    which completes the proof.
\end{proof}

We can now prove a large deviations upper bound for $\left(v_{\lfloor m^{-1}n \rfloor}^{n,m,1},\ldots,v_{\lfloor m^{-1}n \rfloor}^{n,m,m}\right)$.

\begin{proposition}\label{prop:upper_bound_LDP_tangent}
Let the assumptions of Theorem \ref{thm:weighted_GRW_LDP} be satisfied. Denote by $r$ the uniform bound on the increments of the geodesic random walk. Then for $m$ large enough and any closed $F \subset (T_{x_0}M)^m$ we have
\begin{multline}
\limsup_{n\to\infty} \frac1n\log\PP\left(\left(v_{\lfloor m^{-1}n \rfloor}^{n,m,1},\ldots,v_{\lfloor m^{-1}n \rfloor}^{n,m,m}\right) \in F\right)
\\
\leq 
-\inf_{(v_1,\ldots,v_m) \in F} \sup_{(\lambda_1,\ldots,\lambda_m) \in (T_{x_0}M)^m} \frac1m\sum_{i=1}^m \left\{\inp{\lambda_i}{mv_i} - \Psi_{x_0}(\lambda_i) - m^{-\frac12}C|\lambda_i|r^3\right\}.
\end{multline}
Here, $C$ is a constant depending on the curvature of the compact set $\overline{B(0,r)}$ and the bound $r$.
\end{proposition}

\begin{proof}
    From Lemma \ref{lem:distance_RW} it follows that the random walks stay in the compact set $\overline{B(0,r)}$, so that it suffices to prove the statement for compact sets $\Gamma$. By \cite[Theorem 4.5.3]{DZ98} we have
    
    \begin{align*}
    \MoveEqLeft\limsup_{n\to\infty} \frac1n\log\PP\left(\left(v_{\lfloor m^{-1}n \rfloor}^{n,m,1},\ldots,v_{\lfloor m^{-1}n \rfloor}^{n,m,m}\right) \in \Gamma\right)
    \\
    &\leq 
    -\inf_{(v_1,\ldots,v_m) \in \Gamma} \sup_{(\lambda_1,\ldots,\lambda_m) \in (T_{x_0}M)^m} \left\{\sum_{i=1}^m \inp{\lambda_i}{v_i}
    - \limsup_{n\to\infty} \frac1n\log\EE\left(e^{n\sum_{i=1}^m \inp{\lambda_i}{v_{\lfloor m^{-1}n\rfloor}^{n,m,i}}}\right)\right\}.
    \end{align*}

    Together with Proposition \ref{prop:upperbound_MGF} this gives us that for $\sigma$-almost all $\theta$ we have

    \newpage
    
    \begin{align*}
    \MoveEqLeft\limsup_{n\to\infty} \frac1n\log\PP\left(\left(v_{\lfloor m^{-1}n \rfloor}^{n,m,1},\ldots,v_{\lfloor m^{-1}n \rfloor}^{n,m,m}\right) \in \Gamma\right)
    \\
    &\leq 
    -\inf_{(v_1,\ldots,v_m) \in \Gamma} \sup_{(\lambda_1,\ldots,\lambda_m) \in (T_{x_0}M)^m} \left\{\sum_{i=1}^m \inp{\lambda_i}{v_i}
    - \frac1m\sum_{i=1}^m \Psi_{x_0}(\lambda_i) - \frac{Cr^3}{m^{3/2}}\sum_{i=1}^m |\lambda_i|\right\}
    \\
    &= 
    -\inf_{(v_1,\ldots,v_m) \in \Gamma} \sup_{(\lambda_1,\ldots,\lambda_m) \in (T_{x_0}M)^m} \frac1m\sum_{i=1}^m \left\{\inp{\lambda_i}{mv_i}
    - \sum_{i=1}^m \Psi_{x_0}(\lambda_i) - \frac{Cr^3}{m^{1/2}}\sum_{i=1}^m |\lambda_i|\right\}
    \end{align*}
    as desired.
\end{proof}

We are now ready to derive the upper bound of the large deviation principle for $\{\Ss_n^{\theta^n}\}_{n \in \NN}$. For this, we need a suitable map that maps $\left(v_{\lfloor m^{-1}n \rfloor}^{n,m,1},\ldots,v_{\lfloor m^{-1}n \rfloor}^{n,m,m}\right)$ to $\Ss_n^{\theta^n}$. Intuitively, we construct a piecewise geodesic path with the given tangent vectors as directions, which we parallel transport along the constructed path. More precisely, we introduce the map $\Tt_m:(T_{x_0}M)^m \to M$ that constructs this piecewise geodesic path $\gamma:[0,1] \to M$ recursively as follows. Set $\gamma(0) = x_0$ and suppose $\gamma$ has been defined on $\left[0,\frac im\right]$. Then for $t \in \left[\frac im,\frac {(i+1)}m\right]$ we define $\gamma(t) = \Exp_{\gamma(\frac im)}\left(\left(t - \frac im\right)\tau_{\gamma;x_0\gamma(\frac im)}v_{i+1}\right)$. Finally, we set
    \begin{equation}\label{eq:tangent_to_M_upper}
        \Tt_m(v_1,\ldots,v_m) = \gamma(1).
    \end{equation}

Having Proposition \ref{prop:upper_bound_LDP_tangent}, the proof of the large deviations upper bound for $\{\Ss_n^{\theta^n}\}_{n \in \NN}$ is analogous to the proof of \cite[Proposition 6.9]{Ver19}. We provide a condensed version of the proof, emphasizing the adaptations that need to be made for the proof to be valid in our current setting.

\begin{proposition}\label{prop:upper_bound_LDP_theorem}
    Let the assumptions of Theorem \ref{thm:weighted_GRW_LDP} be satisfied. Then for $\sigma$-almost every $\theta$ and every $F \subset M$ closed we have
    $$
    \limsup_{n\to\infty} \frac1n\log \PP\left(\Ss_n^{\theta^n} \in F\right) \leq -\inf_{x\in F} I_M(x),
    $$
    where
    $$
    I_M(x) = \inf\{\Psi_{x_0}^*(v)|v \in \Exp_{x_0}^{-1}x\}.
    $$
\end{proposition}

\begin{proof}
From Proposition \ref{prop:upper_bound_LDP_tangent} and the inequality $|\lambda| \leq |\lambda|^2 + 1$, we obtain
\begin{align*}
\MoveEqLeft\limsup_{n\to\infty} \frac1n\log\PP\left(\Ss_n^{\theta^n} \in F\right)
\\
&\leq 
\frac{Cr^3}{\sqrt m} - \inf_{(v_1,\ldots,v_m) \in \Tt_m^{-1}F} \sup_{(\lambda_1,\ldots,\lambda_m) \in (T_{x_0}M)^m} \frac1m\sum_{i=1}^m \left\{\inp{\lambda_i}{mv_i} - \Psi_{x_0}(\lambda_i) - m^{-1/2}Cr^3|\lambda_i|^2\right\},
\end{align*}

where $\Tt_m$ is as in \eqref{eq:tangent_to_M_upper}. Because $\Lambda_{x_0}$ is differentiable, convex and non-negative, so is $\Psi_{x_0}$. Hence, we can follow the proof of \cite[Proposition 6.9]{Ver19} to get
\begin{multline}
\inf_{(v_1,\ldots,v_m) \in \Tt_m^{-1}F} \sup_{(\lambda_1,\ldots,\lambda_m) \in (T_{x_0}M)^m} \frac1m\sum_{i=1}^m \left\{\inp{\lambda_i}{mv_i} - \Psi_{x_0}(\lambda_i) - m^{-1/2}Cr^3|\lambda_i|^2\right\}
\\
= 
\inf_{v \in \Exp_{x_0}^{-1}F} \sup_{\lambda \in T_{x_0}M}  \left\{\inp{\lambda}{v} - \Psi_{x_0}(\lambda) - m^{-1/2}Cr^2|\lambda|^2\right\}.
\end{multline}

Since 
$$
\inf_{x \in F} I_M(x) = \inf_{v \in \Exp_{x_0}^{-1}F} \sup_{\lambda \in T_{x_0}M} \left\{\inp{\lambda}{v} - \Psi_{x_0}(\lambda)\right\}
$$
it suffices to show that
\begin{equation}\label{eq:limit_upper_bound}
\lim_{m\to\infty} \inf_{v \in \Exp_{x_0}^{-1}F} \sup_{\lambda \in T_{x_0}M}  \left\{\inp{\lambda}{v} - \Psi_{x_0}(\lambda) - m^{-1/2}Cr^3|\lambda|^2\right\}
=
\inf_{v \in \Exp_{x_0}^{-1}F} \sup_{\lambda \in T_{x_0}M} \left\{\inp{\lambda}{v} - \Psi_{x_0}(\lambda)\right\}.
\end{equation}

When we restrict the infimum to $v \in \overline{B(0,2r\EE(|Z|)} \cap \Exp_{x_0}^{-1}F$, this follows from compactness, together with the fact that
$$
\lim_{m\to\infty} \sup_{\lambda \in T_{x_0}M} \left\{\inp{\lambda}{v} - \Psi_{x_0}(\lambda) - m^{-1/2}Cr^3|\lambda|^2\right\} = \sup_{\lambda \in T_{x_0}M} \left\{\inp{\lambda}{v} - \Psi_{x_0}(\lambda)\right\}
$$
because the sequence is increasing in $m$.

For $v \notin \overline{B(0,2r\EE(|Z|)}$, it follows by showing that both sides of \eqref{eq:limit_upper_bound} are infinite. For this, first observe that since the support of of $\mu_{x_0}$ is contained in $\overline{B(0,r)}$, we have 
$$
\Psi_{x_0}(\lambda) \leq \EE(r|Z||\lambda|) = r|\lambda|\EE(|Z|). 
$$

This gives us that
$$
\sup_{\lambda \in T_{x_0}M} \left\{\inp{\lambda}{v} - \Psi_{x_0}(\lambda)\right\} \geq \sup_{\lambda \in T_{x_0}M} \left\{\inp{\lambda}{v} - r\EE(|Z|)|\lambda|\right\} = \infty
$$
for $|v| \geq r\EE(|Z|)$. 

On the other hand,
\begin{align*}
\sup_{\lambda \in T_{x_0}M}  \left\{\inp{\lambda}{v} - \Psi_{x_0}(\lambda) - m^{-1/2}Cr^2|\lambda|^2\right\}
&\geq
\sup_{\lambda \in T_{x_0}M}  \left\{\inp{\lambda}{v} - r|\lambda|\EE(|Z|) - m^{-1/2}Cr^3|\lambda|^2\right\}
\\
&=
\frac{(|v| - r\EE(|Z|))^2\sqrt m}{4Cr^3},
\end{align*}
which tends to infinity as $m\to\infty$, completing the proof.
\end{proof}

\subsection{Lower bound of the large deviation principle}

Similar to the large deviations upper bound, we deduce the lower bound for the large deviation principle for $\{\Ss_n^{\theta^n}\}_{n \in \NN}$ from from a suitable large deviations lower bound in $(T_{x_0}M)^m$. For the upper bound, we had some freedom in choosing how to parallelly transport vectors to eventually map $(T_{x_0}M)^m$ to $M$. Ultimately, we defined the map $\Tt_m$ as in \eqref{eq:tangent_to_M_upper} for this. Unfortunately, this map does not have strong enough continuity properties to transfer a large deviations lower bound from $(T_{x_0}M)^m$ to one for $\{\Ss_n^{\theta^n}\}_{n \in \NN}$. In particular, we require the continuity of the maps in a sense to be uniform in $m$. Therefore, we introduce an adapted version of this map in the next section.

\subsubsection{From $T_{x_0}M$ to $M$}

To prove the lower bound of the large deviation principle, for $x \in M$ and $v \in \Exp_{x_0}^{-1}x \subset T_{x_0}M$ we can focus on realizations of the random walk which stay close to the geodesic $\gamma_v(t) = \Exp_{x_0}(tv)$. We parallelly transport the random increments as much as possible along this geodesic, to limit deviations arising from curvature. 

More specifically, given $m \in \NN$, $v \in T_{x_0}M$, we define a map $\Tt_{v,m}:(T_{x_0}M)^m \to M$ as follows (see also \cite[Section 6.2]{Ver19}). We first discretize the geodesic $\gamma_v(t) = \Exp_{x_0}(tv)$ by setting $x_i := \gamma_v\left(\frac im\right)$, $i = 1,\ldots,m$. Next, we define points $y_i$, $i = 0,\ldots,m$ recursively. First set $y_0 = x_0$ and next suppose $y_k$ is defined. Define $\tilde v_k \in T_{y_k}M$ as $\tilde v_k = \tau_{x_ky_k}\tau_{\gamma_v;x_0x_k}v_k$. Contrary to the setting of the upper bound, here $\tau_{\gamma_v;x_0x_k}$ denotes parallel transport along the geodesic $\gamma_v$, and $\tau_{x_ky_k}$ is parallel transport along any shortest length geodesic connecting $x_k$ and $y_k$. Finally, we set
\begin{equation}\label{eq:tangent_to_M_lower}
\Tt_{v,m}(v_1,\ldots,v_m) := y_m.    
\end{equation}

From \cite[Section 6.2]{Ver19} we have the following, which shows that the continuity of the maps $\Tt_{v,m}$ is in a sense uniform in $m$.

\begin{lemma}\label{lem:continuity}
    Given $v \in T_{x_0}M$ and $\epsilon > 0$, there exists a $\delta > 0$ such that for all $m$ large enough we have
    $$
    \Tt_{v,m}(v_1,\ldots,v_m) \in B(\Exp_{x_0}v,\epsilon)
    $$
    whenever $(v_1,\ldots,v_m) \in B(v,\delta)^m$.
\end{lemma}

\subsubsection{Proof of the large deviations lower bound}

Since we have adapted the map $\Tt_m$ to $\Tt_{v,m}$, we have to adapt the sums $Y_i^n$ as defined in \eqref{eq:Y_upper} accordingly. Let $v \in T_{x_0}M$ and set $\gamma_v(t) = \Exp_{x_0}(tv)$ as in the previous section. Now define for $i = 1,\ldots,m$ the random variables
\begin{equation}\label{eq:incr_rw_tangent}
\bar Y_i^n = \tau_{\gamma_v;x_0\gamma_v(\frac {i-1}m)}^{-1}\tau^{-1}_{\gamma_v(\frac {i-1}m)\Ss_{n_{i-1}}^{\theta^n}}\sum_{k = n_{i-1} + 1}^{n_i} \theta_k^n\tau^{-1}_{\Ss_{n_{i-1}}^{\theta^n}\Ss_{k-1}^{\theta^n}}X_k^n\in T_{x_0}M
\end{equation}

For the same reasons as for $Y_i^n$, $\bar Y_i^n$ is a weighted sum of independent random variables with distribution $\mu_{x_0}$. As a consequence, Corollary \ref{cor:Varadhan_limit_mgf} implies that each of the $\bar Y_i^n$ satisfies a large deviation principle. The continuity property of $\Tt_{v,m}$ as in Lemma \ref{lem:continuity}, especially the uniformity in $m$, then allows us to transfer the corresponding lower bound to the lower bound for the large deviation principle for $\{\Ss_n^{\theta^n}\}_n$.


\begin{proposition}\label{prop:lower_bound_LDP_theorem}
    Let the assumptions of Theorem \ref{thm:weighted_GRW_LDP} be satisfied. Then for $\sigma$-almost every $\theta$ and every $G \subset M$ open we have
    $$
    \limsup_{n\to\infty} \frac1n\log \PP\left(\Ss_n^{\theta^n} \in G\right) \geq -\inf_{x\in G} I_M(x),
    $$
    where
    $$
    I_M(x) = \inf\{\Psi_{x_0}^*(v)|v \in \Exp_{x_0}^{-1}x\}.
    $$
\end{proposition}
\begin{proof}
    It suffices to show that for every $x \in M$ and $\epsilon > 0$ we have
    $$
    \limsup_{n\to\infty} \frac1n\log \PP\left(\Ss_n^{\theta^n} \in B(x,\epsilon)\right) \geq - I_M(x).
    $$
    By definition of $I_M(x)$, it is actually enough to show that
    $$
    \limsup_{n\to\infty} \frac1n\log \PP\left(\Ss_n^{\theta^n} \in B(x,\epsilon)\right) \geq - \Psi^*(v)
    $$
    for all $v \in \Exp_{x_0}^{-1}x$.
    
    Fix such $v$. Write $\gamma_v(t) = \Exp_{x_0}(tv)$ and let $m$ be large enough such that we can uniquely define
    $$
    \tilde v_k^{n,m,i} \in \Exp_{\Ss_{n_{i-1}}^{\theta^n}}^{-1}\left(\Ss_{n_{i-1}+k}^{\theta^n}\right) \subset T_{\Ss_{n_{i-1}}^{\theta^n}}M
    $$
    of minimal length (as in \eqref{eq:pullback_vector_piece}) and set
    $$
    \bar v_k^{n,m,i} = \tau_{\gamma_v;x_0\gamma_v(\frac {i-1}m)}^{-1}\tau^{-1}_{\gamma_v(\frac {i-1}m)\Ss_{n_{i-1}}^{\theta^n}}\tilde v_k^{n,m,i} \in T_{x_0}M. 
    $$
    By construction
    $$
    \Tt_{v,m}(\bar v_{\lfloor m^{-1}n\rfloor}^{n,m,1},\ldots,\bar v_{\lfloor m^{-1}n\rfloor}^{n,m,m}) = \Ss_n^{\theta^n},
    $$
    where $\Tt_{v,m}$ is as in \eqref{eq:tangent_to_M_lower}. By Lemma \ref{lem:continuity} there exists a $\delta > 0$ such that for all $m$ large enough we can estimate
    $$
    \PP\left(\Ss_n^{\theta^n} \in B(x,\epsilon)\right) \geq \PP\left(\left(\bar v_{\lfloor m^{-1}n\rfloor}^{n,m,1},\ldots,\bar v_{\lfloor m^{-1}n\rfloor}^{n,m,m}\right) \in B(v,\delta)^m\right).
    $$

    We compare the latter to $(\bar Y_1^n,\ldots,\bar Y_n^n)$. Proposition \ref{prop:RW_to_tangent_space} and \eqref{eq:CS_coef} imply that
    \begin{align*}
        |\bar v_{\lfloor m^{-1}n\rfloor}^{n,m,i} - \bar Y_i^n|
        &=
        \left|v_{\lfloor m^{-1}n\rfloor}^{n,m,i} - \sum_{k = n_{i-1} + 1}^{n_i} \theta_k^n\tau^{-1}_{\Ss_{n_{i-1}}^{\theta^n}\Ss_{k-1}^{\theta^n}}X_k^n\right|
        \\
        &\leq
        C\frac1n\sum_{k={n_{i-1}+1}}^{n_i} (\theta_k^n)^2 + Cr^3\frac{1}{m^{3/2}}
        \\
        &\leq
        C\frac1n\frac{r}{m^{1/2}} + Cr^3\frac{1}{m^{3/2}}.
    \end{align*}
    As a consequence, we can take $m$ large enough so that
    $$
    |\bar v_{\lfloor m^{-1}n\rfloor}^{n,m,i} - \bar Y_i^n| < \frac{\delta}{2}.
    $$
    From this it follows we can further estimate
    $$
    \PP\left((\bar v_{\lfloor m^{-1}n\rfloor}^{n,m,1},\ldots,\bar v_{\lfloor m^{-1}n\rfloor}^{n,m,m}) \in B(v,\delta)^m\right) \geq \PP\left((\bar Y_1^n,\ldots,\bar Y_m^n) \in B(v,\delta/2)^m\right).
    $$
    Since the increments of the random walk $\Ss_n^{\theta^n}$ are independent, so are the $\bar Y_i^n$. Therefore,
    $$
    \frac1n\log \PP\left((\bar Y_1^n,\ldots,\bar Y_m^n) \in B(v,\delta/2)^m\right) = \frac{1}{n}\sum_{i=1}^m\log \PP\left(\bar Y_i^n \in B(v,\delta/2)\right)
    $$
    Since moreover the distributions of the increments of $\Ss_n^{\theta^n}$ are invariant under parallel transport, it follows that $\bar Y_i^n$ is a sum of independent random variables with distribution $\mu_{x_0}$. As a consequence, it follows from Corollary \ref{cor:Varadhan_limit_mgf} and the Gartner-Ellis theorem (\cite[Theorem 2.3.6]{DZ98}) that $Y_i^n$ satisfies a large deviation principle with rate function $\frac1m\Psi_{x_0}^*$. This implies that
    $$
    \liminf_{n\to\infty} \frac 1n \sum_{i=1}^m\log \PP\left(\bar Y_i^n \in B(v,\delta/2)\right) \geq - \sum_{i=1}^m \frac1m\Psi_{x_0}^*(v) = -\Psi_{x_0}^*(v),
    $$
    which completes the proof.
    \end{proof}

\printbibliography

\end{document}